\documentclass[12pt]{amsproc}
\usepackage{amsmath}
\usepackage{amsthm}
\usepackage{amssymb}
\usepackage{framed}
\usepackage[hidelinks]{hyperref}
\usepackage{xcolor}
\usepackage{graphicx}
\usepackage{enumerate}
\usepackage[a4paper,  margin=2.5cm]{geometry}

\newtheorem{de}{Definition}

\newtheorem{pr}{Proposition}

\theoremstyle{definition}

\newcommand{\dif}{\mathrm{d}}
\newcommand{\ii}{\mathrm{i}}

\newcommand{\abs}[1]{\vert #1 \vert}  
\newcommand{\ov}{\overline}
\newcommand{\gr}{\mathrm{grad}}
\newcommand{\Ker}{\mathrm{Ker}\,}

\newcommand{\CC}{\mathbb{C}}

\newcommand{\RR}{\mathbb{R}}   
\newcommand{\ZZ}{\mathbb{Z}} 

\newcommand{\Ss}{\mathbb{S}}

\DeclareMathOperator{\di}{div}
\DeclareMathOperator{\cu}{curl}

\begin{document}

\title[Steady Euler flows on $\Ss^3$ $\&$ Faddeev-Skyrme solutions]{A steady Euler flow on the 3-sphere and its associated Faddeev-Skyrme solution}

% all of whose flow lines are closed and linked twice}

\author{R. Slobodeanu}

\address{Faculty of Physics, University of Bucharest, P.O. Box Mg-11, Bucharest-M\u agurele, RO-077125, Romania}

\email{\texttt{radualexandru.slobodeanu@g.unibuc.ro}}

\date{\today}

\thanks{I thank D. Harland and D. Peralta-Salas for  several important hints.}

\subjclass[2010]{53C25, 58J50, 35Q31, 74G65.}

\keywords{Steady fluid, 3-sphere, linking, Hopf invariant, Faddeev-Skyrme model.}

\maketitle

\begin{center}
\textit{Dedicated with gratitude to Vasile Br\^inz\u anescu on his 75th birthday}
\end{center}

\medskip

\begin{abstract}
We present a steady Euler flow on the round 3-sphere $\Ss^3$ whose velocity vector field has the remarkable property of having two independent first integrals, being tangent to the fibres of an almost submersion onto the 2-sphere. This submersion turns out to be a critical point for the quartic Faddeev-Skyrme model with a standard potential.
\end{abstract}

\section{Introduction}
The now obsolete idea of vortex atoms (stable knotted thin vortex tubes in ether) proposed by Lord Kelvin in 1867 became later the main motivation for developping the knot theory and recently have registered a resurgence in mathematical physics \textit{via} some related ideas. Knotted and linked stream/vortex lines in fluid dynamics (or magnetic field lines in magnetohydrodynamics) and stable knot solitons in field theory (and also in condensed matter, chiral ferromagnetic liquid crystal colloids, and other areas \cite{sutclif}) are the subject of many impactful and mathematically deep recent studies. In the first case we work with (steady) vector field solutions in $\RR^3$ for Euler's fluid equations, while in the latter, with mapping solutions (from $\RR^3$ into the 2-sphere $\Ss^2$) for the Euler-Lagrange equations of Faddeev's reduction \cite{fad0} of the Skyrme quartic energy functional. It turns out that \cite{slo} each stationary solution $\varphi: M \to \Ss^2$ of the Faddeev-Skyrme model gives rise to a steady (forced) Euler fluid with vector field $V$ belonging to $\Ker(\dif \varphi)$, and conversely, each steady Euler vector field solution can be integrated \textit{locally} to a Faddeev-Skyrme solution (here $M$ may be a different smooth 3-manifold than $\RR^3$). To mark the special case where the converse holds globally we coined the name $S$-integrable for the respective steady Euler field. This correspondence, that has also a relativistic version \cite{slorel}, suggests that problems related to (the knotted solitons in) the Faddeev-Skyrme model are more constrained than their analogues for Euler flows. Indeed, while for steady Euler equations many important results are known (as the existence of Beltrami type solutions with periodic stream/vortex lines diffeomorphic to any given link \cite{EP1, enc}, or with thin vortex tubes of any link and knot type \cite{EP2}, and the existence of compactly supported solutions in $\RR^3$ \cite{gav}, to cite only few recent milestones), for Faddeev-Skyrme equations the existence results are sparse \cite{lin} and most investigations relies on numerics.

In this paper we supplement the picture in \cite{slo} with a new example of $S$-integrable steady Euler flow on the round 3-sphere $\Ss^3$ all of whose flow-lines are closed and linked twice, and that corresponds to a Faddeev-Skyrme solution (for the energy comprising only the quartic term and a standard potential) of Hopf invariant 2. In the next section we fix the set-up giving the main definitions and the precise formulation of the questions of interest. In section $\S 3$ we present our solution and its main properties, and in the Appendix we give some computational details and explain some conventions.

\section{$S$-integrable Euler fields and Faddeev-Skyrme solutions}

In this section we recall some definitions and facts about Euler's fluid equations and about Faddeev-Skyrme 
$\sigma$-model that will be useful in the sequel and we make precise the questions of interest. For more informations on topological fluid dynamics see the monography \cite{arn} and the introductory paper \cite{seltop}.

\begin{de}
A \emph{steady Euler field} on a Riemannian $3$-manifold $(M, g)$ is a tangent vector field $V$ on $M$ which is a solution of the \emph{stationary Euler equations}
\begin{equation} \label{eul}
\nabla_V V = - \gr \, p, \qquad \di V = 0
\end{equation}
for some pressure function $p$ on $M$.
\end{de}
When $M$ is (a domain of) the euclidean space $\RR^3$, the couple $(V, p)$ models an incompressible inviscid fluid (of constant density) in equilibrium. The associated flow will be called steady Euler (fluid) flow.
An equivalent reformulation of \eqref{eul} is
\begin{equation} \label{eulb}
V \times \cu V = \gr \, b, \qquad \di V = 0,
\end{equation}
with $b=p+\frac{1}{2}\abs{V}^2$ the Bernoulli function and $\cu V = (\ast \dif V^\flat)^\sharp$ the vorticity field, where $\ast$ is the Hodge star operator on $(M,g)$. Notice that $b$ is conserved along the flow: $V(b)=0$, i.e. when not constant, $b$ is a first integral of $V$.
 
From \eqref{eulb} we see that a divergence-free vector aligned with its own curl is a particular type of steady Euler field (with constant $b$). This solutions are usually called \textit{Beltrami fields} (aka \textit{force-free fields} in MHD). When the proportionality factor is constant we emphasize it by calling them \textit{strong} Beltrami fields (or simply curl-eigenvectors).

A very special class of solutions is obtained when the vector field $V \in \Gamma(TM)$ is \textit{completely integrable} (i.e. it has 2 independent first integrals; see e.g. \cite{miy, completei} for some results about this type of fields in the Euclidean space). In \cite{slo} we introduced the related notion of $S$-integrability as follows: a vector field $V$ on a $3$-manifold $M$ is \emph{$S$-integrable} if it exists a smooth (at least $C^2$) map $\varphi: M \to N$ to some surface $N$, such that $V_x \in \ker (\dif \varphi_x)$ at any regular point $x$ of $\varphi$ and $V \big \vert _{C_\varphi}=0$, where $C_\varphi = \{ x \in M : \mathrm{rank}(\dif \varphi_x) < 2\}$ is the critical set of $\varphi$. When $N$ is precisely the round 2-sphere we'll emphasize it by employing the terminology \emph{$\Ss$-integrable}.

The dynamical features of $S$-integrable steady Euler fields $V$ follow from the fact the their generic orbits (or stream lines) coincide to the regular fibres of the associated map $\varphi$ (here by regular fibre we mean the preimage of a regular point of $\varphi$). If $\varphi$ is smooth, by the regular level set theorem, any regular fibre is a 1-dimensional regular submanifold, so in particular a generic orbit of the vector field $V$ cannot be quasi-periodic (dense in a torus). Using the constant-rank level set theorem on a sufficiently small neighbourhood of a regular point, one can also exclude the possibility of having quasi-periodic orbits of $V$ inside the preimage of a critical value of $\varphi$.  Thus, in the case of $S$-integrable flows, we cannot encounter this type of dynamics with dense orbits (as it is the case for e.g. the flow in \cite{gav}). If moreover $M$ is closed, then all generic orbits of $V$ are periodic (closed curves) whose linking number can be related \cite{arn0, arn} to the Hopf homotopy invariant \footnote{This invariant can be defined if $H^2(M, \ZZ)=0$ or if we restrict to algebraically inessential maps.} $Q(\varphi) \in \ZZ$ of the associated map $\varphi: M \to \Ss^2$; in addition, by Arnold's structure theorem \cite{arn}, the orbits of $V$ lie on invariant 2-tori (which are the regular levels sets of the Bernoulli function $b: M \to \RR$, assumed to be non-constant).

\medskip

As they are extremely non-generic, $S$-integrable steady Euler flows are undoubtedly hard to find. In this note we start with a given almost submersion and then we check whether a suitable vector field tangent to the fibres is a steady Euler field, that will be $S$-integrable by construction.

It is known \cite[Prop.3]{sl} that, given a smooth map  $\varphi : (M^3, g) \to (N^2, h)$ between Riemannian manifolds (exponents indicate the dimension), the vector field $V = \lambda_1 \lambda_2 U$ (locally defined around a regular point of $\varphi$) is divergence-free, where $\lambda_i^2$ are the eigenvalues of $\varphi^* h$ w.r.t. $g$ and $U$ is a unit vector spanning $\ker(\dif \varphi)$.
Alternatively (see \cite{slo}) if $\omega$ is the area 2-form induced by $h$ on $N$, then $V=(\ast \varphi^* \omega)^\sharp$ is a divergence-free vector field (notice that this statement is immediate when $H^2(M)=0$ since the closed 2-form $\varphi^* \omega$ is then exact, so that $V$ is the $\cu$ of some vector field on $M$ \footnote{With the terminology in \cite{che}, if $N=\Ss^2$, then $\varphi$ is a \textit{spherical Clebsch map} representing $V$. This method of constructing divergenceless fields was noticed also in \cite{as}.}). Enjoying the divergenceless property by construction, such a vector field  $V$ will be a steady ($S$-integrable) Euler field if moreover the first equation in \eqref{eulb} is satisfied. When $H^1(M)=0$ it is then necessary and sufficient to check the following identity:
\begin{equation}\label{commute}
[V, \cu V]=0.
\end{equation} 
To see this, simply recall the identity: $\cu(A \times B) = (\di B)A - (\di A)B - [A,B]$.

\medskip

According to the main result \cite{slo}, if this programme succeeds, then $\varphi$ is a $\sigma_2$-critical map with potential $P$ (given by the Bernoulli function), that is a solution for the Faddeev-Skyrme model. Let us recall the following
\begin{de}
A smooth (at least $C^2$) map $\varphi: M \to \CC P^1 \cong \Ss^2(\frac{1}{2})$ defined on a Riemannian $3$-manifold $(M,g)$ with volume element $\upsilon_g$ is called a (classical) \emph{Faddeev-Skyrme solution} if it is a critical point of the following energy:
\begin{equation} \label{SFpotFull}
E(\varphi) = \frac{1}{2} \int_{M} \{\alpha_2 \abs{\dif \varphi}^2+ \alpha_4 \abs{\varphi^* \omega}^2 + 2\alpha_0 P(\varphi)\} \upsilon_g, 
\end{equation}
with $\alpha_0$, $\alpha_2$ and $\alpha_4$ non-negative real parameters and $P$ a non-negative real function on $\Ss^2$. If $\alpha_2=0$ and  $\alpha_0=\alpha_4=1$, we also called \cite{slo} such a map a \emph{$\sigma_2$-critical map with potential} $P$.
\end{de}

Although the original model is defined for $M=\RR^3$ and with $\alpha_0=0$ \cite{fad}, the interest was spread to the case of compact manifolds $M$ \cite{ada, fer, har, spei, ward} and to the case with potential ($\alpha_0 \neq 0$) \cite{ada, fos, nit, shn}. In the case of $M=\RR^3$ (where we impose $\lim_{\abs{x} \to \infty}\varphi(x)=(0,0,1)$) and of $M=\Ss^3$, since $\pi_3(\Ss^2) \cong \ZZ$, the homotopy classes of maps $\varphi: M \to \Ss^2$ are indexed by the \emph{Hopf invariant} $Q(\varphi) = \frac{1}{(\int_{\Ss^2} \omega)^2}\int_{\Ss^3} \alpha \wedge \varphi^* \omega$, where $\omega$ is an area form on the codomain and $\alpha$ is any 1-form satisfying $\dif \alpha = \varphi^* \omega$. The energy \eqref{SFpotFull} admits a lower bound in terms of some power of $Q(\varphi)$, as for instance \eqref{lboundfreedman} (see \cite{har} for various lower bounds depending on which of the parameters $\alpha_i$ vanishes). Therefore the main problem is to find stable finite energy solutions in each homotopy class, that will represent \textit{topological solitons} with knotted position curve $\varphi^{-1}(0,0,-1)$. These solutions are also known as \textit{hopfions} \cite{sutclif}. To spot the place of hopfions in the broader context of topological solitons in classical field theory, see \cite{man}.

In view of the "duality" with steady Euler flows when $\alpha_2=0$ (or with forced Euler flows when $\alpha_2 \neq 0$), the question of finding Faddeev-Skyrme solutions on $\Ss^3$ with given Hopf invariant $k$ can be reformulated as:

\medskip
\noindent \textbf{Question}: Does there exist for each $k\in \ZZ$ an $\Ss$-integrable steady Euler solution on $\Ss^3$ whose generic stream lines are (all closed and) linked $k$ times?

\medskip
This has been answered in the afirmative in \cite{slo} if one allows the associated submersion $\varphi$ to be only of class $C^1$ (except for $k=1,2$ where smoothness is achieved) and if the metric on $\Ss^3$ can be chosen freely (depending on $k$). All associated steady Euler flows belongs to a class of solutions discovered in \cite{kam} that will be called \textit{KKPS solutions} in the sequel.

In this note we'll impose that the metric is the standard one and that $\varphi$ is smooth. Under this requirements, the only known solutions are those in \cite[Example 2]{slo} corresponding linking numbers that are perfect squares ($k^2$). In this note we answer the above question for $k=2$ (actually our new solution is even analytic).

\medskip
We mention also a related stronger version of the above question (corresponding to Faddeev-Skyrme solutions with given Hopf invariant, when $\alpha_0=\alpha_2=0$):

\medskip
\noindent \textbf{Question$_B$}: Does there exist for each $k\in \ZZ$ an $\Ss$-integrable (strong) Beltrami field on the round unit 3-sphere whose generic stream lines are (all closed and) linked $k$ times?

\medskip
For this question the only known example with smooth associated map $\varphi$ is the (anti-) Hopf field ($k=\pm 1$). Beside this classic example, $\Ss$-integrable (in a weak sense) non-vanishing Beltrami fields with non-constant proportionality factor and arbitrary linking number can be deduced from the Faddeev-Skyrme solution in \cite{fer, slobo}, but the associated map $\varphi$ has two circles of singular points (where $\dif \varphi$ fails to be continuous). See also \cite[$\S 6.1$]{radu} for energy minimizing Beltrami fields (corresponding to hopfions) with arbitrary $k$, where the metric on $\Ss^3$ was suitably chosen (depending on $k$) and the codomain $N$ is allowed to be a 2-orbifold with two conical singularities (the weighted projective space). 

Notice that in the context of \textit{Question}$_B$ the associated map $\varphi$ must be an almost submersion, i.e. a submersion on a dense set, since the zero set of a Beltrami field is nowhere dense (as Beltrami fields enjoy the unique continuation property) and it must coincide with $C_\varphi$.  

\section{A new solution}

To illustrate the idea in the previous section, let as consider as "test submersion" the Hopf map 
\begin{equation}\label{hop}
\pi  : \Ss^3 \to \Ss^2 , \quad \pi(x_1, y_1, x_2, y_2) = (2(x_1 x_2 + y_1 y_2), 2(x_1 y_2 - x_2 y_1), x_1^2 + y_1^2 - x_2^2 - y_2^2)
\end{equation}
composed with the quadratic map of topological degree 2 (cf. \cite{woo})
\begin{equation}\label{quadra}
\psi : \Ss^3 \to \Ss^3, \quad  \psi(x_1, y_1, x_2, y_2) = (x_1^2 - y_1^2 - x_2^2 - y_2^2, 2x_1 y_1, 2x_1 x_2, 2x_1 y_2).
\end{equation}
It is well-known that $\varphi=\pi \circ \psi$ is then a map of Hopf invariant $Q(\varphi)=\mathrm{deg}(\psi)Q(\pi)=2$. 

We are ready now to formulate our main result:

\begin{pr}
Let $\varphi=\pi \circ \psi : \Ss^3 \to \Ss^2(\tfrac{1}{2})$ be the mapping defined above, and consider the spheres endowed with the usual round metrics (with $\omega$ the associated area 2-form on the codomain). Then

\noindent $(i)$ \ the vector field $V=(\ast \varphi^* \omega)^\sharp$ is an $\Ss$-integrable steady Euler flow on $\mathbb{S}^3$ with the associated Bernoulli function: $b=8x_1^2(x_2^2 + y_2^2)$ and pressure $p=-8x_1^4$. Its explicit expression with respect to the standard orthonormal (global) frame \eqref{stdframe} on $\Ss^3$ is: 
\begin{equation} \label{Sint}
V=4x_1(x_1 \, \xi - y_2 \, X_1 +  x_2 \, X_2); 
\end{equation}

\noindent $(ii)$ \ the almost subersion $\varphi$ is a smooth critical point of Hopf invariant $Q(\varphi)=2$ for the quartic Faddeev-Skyrme energy
\begin{equation} \label{SFpot}
E(\varphi) = \frac{1}{2} \int_{\Ss^3} \{\abs{\varphi^* \omega}^2 + 2 (1- \varphi_3)\} \upsilon_g.
\end{equation}
with the potential $P(\varphi)= 1 - \varphi_3$, where $\varphi_3$ is the $3^{rd}$ component of $\varphi$ seen as map to $\RR^3$.
\end{pr}

\begin{proof} $(i)$ First we have to construct the vector field $V=(\ast \varphi^* \omega)^\sharp$ then to verify by direct computation that \eqref{commute} holds true. Henceforth the pressure and Bernoulli function are found by integration (we have used Mathematica \cite{math}, but all computations may be also done by hand. The Mathematica worksheet with fully detailed proofs is available from the author). In the computation of $V$ it is useful to remark that $\pi^* \omega = \frac{1}{2}\dif \eta$, where $\eta$ is the contact form dual to the Reeb field $\xi$ on $\Ss^3$ given explicitly in \eqref{stdframe}. Therefore $V=\frac{1}{2}(\ast \dif \psi^* \eta)^\sharp =\frac{1}{2}\cu \widehat{\xi}$, with $\widehat{\xi}:= (\psi^* \eta)^\sharp$, and for the computation of $\cu \widehat{\xi}$ one can use \cite[(2.6)]{radu}. We can check that $V_x \in \ker (\dif \varphi_x)$ at any regular point $x$, so that $V$ is $\Ss$-integrable (see also item $(v)$ below).

\noindent $(ii)$ We verify by direct computation that $b = 1 - \varphi_3$, then apply \cite[Proposition 2]{slo}.
\end{proof}

\medskip

In the rest of this section we list some properties of the steady Euler flow $V$ / of the Faddeev-Skyrme solution $\varphi$.
 
\begin{enumerate}[(i)]
\item As the Hopf map itself, the mapping $\varphi$ defined above is a "happy accident": they are the only Faddeev-Skyrme solutions (critical points of the energy in \eqref{SFpotFull} with $\alpha_2=0$ and $\alpha_0=\alpha_4=1$, for \textit{any} choice of potential) in a series of polynomial mappings that can be constructed in each homotopy class using the same recipe $\varphi_k = \pi \circ \psi_k$, where the degree $k$ map $\psi_k:\Ss^3 \to \Ss^3$ is defined recurently as (see \cite{woo}):
\begin{equation}\label{mod4}
\begin{split}
\psi_{4k}   & = -\psi_2 + 2 \langle \psi_2, \psi_{2k+1} \rangle \psi_{2k+1}\\
\psi_{4k+1} & = -\psi_1 + 2 \langle \psi_1, \psi_{2k+1} \rangle \psi_{2k+1}\\
\psi_{4k+2} & = -\psi_2 + 2 \langle \psi_2, \psi_{2k+2} \rangle \psi_{2k+2}\\
\psi_{4k+3} & = -\psi_1 + 2 \langle \psi_1, \psi_{2k+2} \rangle \psi_{2k+2}
\end{split}
\end{equation}
with $\psi_1 := \mathrm{Id}_{\Ss^3}$, the identity map, and $\psi_2:=\psi$ given by Equation \eqref{quadra}. If we see $\Ss^3$ as the space of unit quaternions,  $\psi_k$ takes the compact form $\psi_k(q)=q^k$ (but the expanded form \eqref{mod4} is necessary in order to be able to apply \cite[Lemma 5]{woo} and to obtain $\mathrm{deg}(\psi_k) = k$). In spherical coordinates (see Appendix) the mappings $\psi_k$ all have the same nice simple form reminiscent of the \textit{hedgehog map} \eqref{hedge}
\begin{equation}\label{sphersimple}
\psi_k: \ (\cos s, \sin s(\cos t, \sin t \, e^{\ii \chi})) \mapsto (\cos ks, \sin ks(\cos t, \sin t \, e^{\ii \chi})).
\end{equation}

The above statement of unicity can be simply proved by computing the vector field $V_k=(\ast \varphi_k^* \omega)^\sharp$ and then by direct verification that \eqref{commute} is satisfied only for $k=1,2$. Actually one can prove a stronger fact: $\varphi_{1,2}$ are the only  solutions (for $\alpha_2=0$ and any choice of potential) within the \textit{ansatz} $\pi \circ \psi_\beta$ with $\psi_\beta (\cos s, \sin s(\cos t, \sin t \, e^{\ii \chi}))=(\cos \beta(s), \sin \beta(s)(\cos t, \sin t \, e^{\ii \chi}))$.

The sequence $\varphi_k$ is of course only one of the possible ways of browsing the homotopy classes with polynomial maps. It would be interesting to have a systematic approach able to pinpoint \textit{all} polynomial Faddeev-Skyrme solutions (for some potential), and their corresponding $\Ss$-integrable steady Euler flows. \\

\item The smooth Faddeev-Skyrme solution $\varphi$ (of the variational problem associated to \eqref{SFpot}) is a mapping which does not belong to the \textit{ansatzes} considered usually in the literature as initial configurations for the (numerical) energy minimization schemes. For instance $\varphi$ does not belong to the \textit{standard ansatz} (known to geometers as \textit{$\alpha$-Hopf construction} \cite{eer}), since in Hopf coordinates (see Appendix) it can be described as following ($\alpha$, $\beta$ can be explicitly given in terms of trigonometric functions):
\begin{equation}
(\cos s \, e^{\ii \phi_1}, \sin s \, e^{\ii \phi_2}) \mapsto \left(\cos \alpha(s, \phi_1), \ \sin \alpha(s, \phi_1)e^{\ii (\beta(s, \phi_1)+\phi_2)}\right).
\end{equation}

Neither does it belong to the \textit{rational map ansatz} \cite{sut} since composing the stereographic projection (from South pole) with $\varphi$ yields the following complex function:
$$
(z_1, z_2) \mapsto \frac{\varphi_1 + \ii \varphi_2}{1+\varphi_3}=\frac{4 \mathrm{Re}(z_1) z_2 (\ov{z}_1^2 - \abs{z_2}^2)}{1 + \abs{z_1}^4 + \abs{z_2}^4 - 2 \abs{z_2}^2 (\ov{z}_1^2 + 2 \mathrm{Re}(z_1) \, z_1)}
$$
For comparison, the complex function associated to Hopf map reads $(z_1, z_2) \mapsto \frac{\pi_1 + \ii \pi_2}{1+\pi_3}=\frac{z_2}{z_1}$ and represents the simplest member of the rational map ansatz. \\

\item The topological lower bound in \cite{har} for the Faddeev-Skyrme energy \eqref{SFpotFull} with $\alpha_2=0$ and  $\alpha_0=\alpha_4=1$ (quartic energy + potential) of mappings $\varphi:\RR^3 \to \Ss^2$ holds true also for mappings $\Ss^3\to \Ss^2(\frac{1}{2})$ and reads:
\begin{equation}\label{lboundfreedman}
E(\varphi) \geq \frac{4}{(27\pi)^{1/4}}\left(\int_{\Ss^2(\frac{1}{2})} \Omega\right)^{3/2} \abs{Q(\varphi)}^{\frac{3}{4}},
\end{equation}
where $\Omega=(2P(\varphi))^{1/6} \omega$, with $\omega$ the standard area form on $\Ss^2(\frac{1}{2})$. To see this we simply have to remark that the first part of the proof in \cite[$\S 4$]{har} does not depend on the domain of definition of $\varphi$, then, at the point where one uses the $L^{3/2}(\RR^3)$-energy estimate from \cite{free}, one should remark that this energy is conformally invariant in dimension 3 so it also holds on $\Ss^3$.

In our case $E(\varphi) = 46.058$, much higher than the right hand term which is approximately 13.852. At the same time, it is not clear whether \eqref{lboundfreedman} can be attained by a smooth map.

Notice that \eqref{lboundfreedman} implies a lower bound of the quantity $\int \{\abs{V}^2 + 2 b \} \upsilon_g$ defined for a $S$-integrable steady Euler field (and its Bernoulli function), in terms of helicity. \\

\item Transporting the Faddeev-Skyrme solution $\varphi$ (or the steady Euler field $V$) to $\RR^3$ fails if we consider the inverse stereographic projection $\sigma^{-1}$ or the "hedgehog map" $H: \RR^3 \to \Ss^3$ defined as
\begin{equation}\label{hedge}
r(\cos \theta, \, \sin \theta e^{\ii \phi}) \mapsto \left(\cos f(r), \, \sin f(r)(\cos \theta, \, \sin \theta  e^{\ii \phi})\right), \ f(0)=\pi, \  f(\infty)=0.
\end{equation}
By this, we mean that the composed map $\varphi \circ H$ (or $\varphi \circ \sigma^{-1}$) is no longer a Faddeev-Skyrme solution and the corresponding $V$ on $\RR^3$ is not a steady Euler field (but it is, by construction, a completely integrable divergence-free non-axisymetric vector field!). It would be a major achievement to find a suitable map $\RR^3 \to \Ss^3$ that pulls a Faddeev-Skyrme solution on $\Ss^3$  back to a solution defined on the Euclidean space. At this point mappings like $\varphi \circ H$ or $\varphi \circ \sigma^{-1}$ may be useful only as initial configurations for numerical energy minimization. \\

\item The critical set $C_\varphi$ of $\varphi$ is equal to the equatorial 2-sphere obtained as the intersection of $\Ss^3$ with the hyperplane $x_1=0$. In Hopf coordinates we can describe it as $C_\varphi = \{(\cos s \, e^{\ii \frac{(2m+1) \pi}{2}}, \, \sin s\, e^{\ii\phi_2})  : s \in [0, \pi/2], \phi_2 \in [0, 2\pi), m=0,1\}$. In particular $\varphi$ is a submersion on a dense set, i.e. $\varphi : \Ss^3 \setminus C_\varphi \to \Ss^2(\frac{1}{2})$ is a submersion. As expected, since $\abs{\varphi^* \omega}=\abs{V}$,  the critical set $C_\varphi$ coincides to the zeros of the steady Euler field $V$ in \eqref{Sint}.

The only critical value of $\varphi$ is the "North pole" $(0,0,1) \in \Ss^2$. Its preimage contains also regular points on a circle that intersect the critical set in two points (so the two semicircles represent heteroclinic orbits for $V$): $\varphi^{-1}(0,0,1)=C_\varphi \cup \{(e^{i\phi_1},0)  : \phi_1 \in [0, 2\pi)\}$. In particular, at any point $C_\varphi$ the Faddeev-Skyrme energy density $\abs{\varphi^* \omega}^2 + 2(1- \varphi_3)$ is zero.

The entire "Hopf link" $\pi^{-1}(0,0,\pm 1)=\{(e^{i\phi_1},0)  : \phi_1 \in [0, 2\pi)\}\cup \{(0, \, e^{i\phi_2})  : \phi_2 \in [0, 2\pi)\}$ is sent to the North pole. Notice that $V$ vanishes along the second circle of the Hopf link. The analogue of the Hopf link in our case (ignoring $C_\varphi$) is $\Gamma_1 \cup \Gamma_2$ with $\Gamma_1=\varphi^{-1}(0,0, 1) \setminus C_\varphi = \{(e^{i\phi_1},0)  : \phi_1 \in [0, 2\pi)\}$ and $\Gamma_2=\varphi^{-1}(0,0,- 1)  = \{(\frac{1}{\sqrt{2}}e^{\ii m \pi}, \, \frac{1}{\sqrt{2}}e^{\ii\phi_2})  : \phi_2 \in [0, 2\pi), m=0,1\}$ and we can "see" that $\Gamma_2$ intersects twice the half 2-sphere bounded by $\Gamma_1$ so that their linking number is 2 (recall that in our context the Hopf invariant of $\varphi$ equals the linking number of any two \textit{regular} fibres, which are also generic orbits of the Euler field $V$).\\

\item The critical set of the Bernoulli function $b$ is given by $C_b = C_\varphi \cup \{(e^{i\phi_1},0)  : \phi_1 \in [0, 2\pi)\}\cup \{(\frac{1}{\sqrt{2}}e^{\ii m \pi}, \, \frac{1}{\sqrt{2}}e^{\ii\phi_2})  : \phi_2 \in [0, 2\pi), m=0,1\}$   and its critical values are 0 and 2. Accordingly, the regular Bernoulli level surfaces ($b^{-1}(\varsigma)$, with $\varsigma \neq 0,2$) to which the dynamic of $V$ is confined are (diffeomorphic) 2-tori \cite[Theorem 1.10]{arn}. On each torus the motion along integral curves of $V$ is \textit{periodic}. 

While in the case of KKPS solutions \cite{kam} the regular level sets of the Bernoulli function form a parallel family of 2-tori of constant mean curvature in $\Ss^3$ (since $b$ is isoparametric, as $V=F(\cos^2 s) \partial_{\phi_1}+G(\cos^2 s) \partial_{\phi_2}$ and $b=b(\cos^2 s)$), for our solution \eqref{Sint} this does not hold any more. 

Besides the Bernoulli function $b=1-\varphi_3$, another (independent) first integral for $V$ is $\varphi_1$ (or $\varphi_2$) since $V \in \ker(\dif \varphi)$ so $V(\varphi_1)=\dif (\imath \circ \varphi)(V)(w_1)=\dif \varphi(V)(w_1\circ\imath)=0$, where $\imath$ is the inclusion map and $w_i$ the coordinates functions on $\RR^3$.

\noindent Notice also that at the zeros of $V$ (\textit{equilibrium points}) both $b$ and $p$ vanish. \\

\item There is no isometry relating our steady Euler flow to a KKPS type flow, since $V(\abs{V})\neq 0$ almost everywhere for our Euler field, while for any KKPS field we have $V(\abs{V}) = 0$ and this is a property conserved by isometries. \\

\item Recall \cite{arn0, arn} that "helicity bounds energy" for any exact divergence-free vector field:
\begin{equation}\label{hbe}
\int_{\Ss^3} \abs{V}^2 \upsilon_g \geq 2 \mathcal{H}(V), \qquad \mathcal{H}(V):=\big(\cu^{-1}V, \, V\big)_{L^2}.
\end{equation}
In our setup, the helicity of $V$ is related to the Hopf invariant of $\varphi$ through the simple relation $\mathcal{H}(V)=\pi^2 Q(\varphi)$. For the steady Euler field \eqref{Sint} the $L^2$ energy equals $\frac{20}{3}\pi^2$, so is much higher than its lower bound $2\mathcal{H}(V)=4\pi^2$ which is invariant in the orbit of $V$ through the action of volume preserving diffeomorphisms.
\end{enumerate}

\section{Appendix}

\subsection{Cartesian coordinates.}
The sphere $\Ss^3$ is seen as the set of points $(z_1, z_2) \in \CC^2$ with $\abs{z_1}^2 + \abs{z_2}^2 =1$. Denoting $z_j= x_j + \ii y_j$, at each point $(x_1, y_1, x_2, y_2) \in \Ss^3$ we have the orthonormal frame (of Killing vector fields that are moreover eigenvectors of $\cu$ for the first positive eigenvalue $\mu_1=2$):
\begin{equation} \label{stdframe}
\begin{split}
\xi     & = - y_1 \partial_{x_1} + x_1 \partial_{y_1} - y_2 \partial_{x_2} + x_2 \partial_{y_2} ,\\
X_1 & =-x_2 \partial_{x_1} + y_2 \partial_{y_1} + x_1 \partial_{x_2} - y_1 \partial_{y_2} ,\\
X_2 & =-y_2 \partial_{x_1} - x_2 \partial_{y_1} + y_1 \partial_{x_2} + x_1 \partial_{y_2}.
\end{split}
\end{equation}

\subsection{Hopf coordinates.}
$(x_1, y_1, x_2, y_2)=(\cos s \, e^{i\phi_1}, \sin s \, e^{i\phi_2})$, $s \in [0, \pi/2]$, $\phi_i \in [0, 2\pi)$.
The induced metric reads $g=\dif s^2 + \cos ^2 s \, \dif \phi_1^2 + \sin ^2 s \, \dif \phi_2^2$ and the standard orthonormal frame above becomes
\begin{equation} \label{hopf}
\begin{split}
\xi &= \partial_{\phi_1} + \partial_{\phi_2}, \\
X_1 &= \cos(\phi_1 + \phi_2)\partial_s +\sin(\phi_1 + \phi_2)(\tan s \, \partial_{\phi_1} - \cot s \, \partial_{\phi_2}), \\
X_2 &= \sin(\phi_1 + \phi_2)\partial_s -\cos(\phi_1 + \phi_2)(\tan s \, \partial_{\phi_1} - \cot s \, \partial_{\phi_2}) ,
\end{split}
\end{equation}

\noindent In these coordinates the Hopf map \eqref{hop} reads 
$$\pi(\cos s \, e^{i\phi_1}, \sin s \, e^{i\phi_2})=(\sin 2s \, e^{\ii(-\phi_1 +\phi_2)}, \ \cos 2s)$$ and the contact form $\eta = \xi^\flat = \cos^2 s \dif \phi_1 + \sin^2 s \dif \phi_2$. On $\Ss^2(\frac{1}{2})$ the standard area form is $\omega = \frac{1}{4}\sin u \dif u \wedge \dif v$ in the chart $(\sin u e^{\ii v}, \, \cos u)$  around North pole. Then we can directly check $\pi^* \omega = \frac{1}{2}\dif \eta$ so the Hopf invariant is  $Q(\pi) = \frac{1}{\mathrm{area}(\Ss^2(\frac{1}{2}))^2}\int_{\Ss^3} \frac{1}{4}\eta \wedge \dif \eta = \frac{1}{\pi^2}\frac{\mathrm{vol}(\Ss^3)}{2}=  1$, since the orientation on $\Ss^3$ is given by the (metric) volume element $\upsilon_g= \frac{1}{2}\eta \wedge \dif \eta$ (so that $\ast \dif \eta = 2 \eta$).

\noindent The solution we found reads: 
\begin{equation} \label{SintHopf}
V=\sin 2 s \sin 2 \phi_1 \, \partial_s + 4 \cos 2s \cos^2 \phi_1 \, \partial_{\phi_1} +  8 \cos^2 s \cos^2 \phi_1 \, \partial_{\phi_2}
\end{equation}
with the associated Bernoulli function: $b=2\sin^2 2s \cos^2 \phi_1$ and pressure $p=-8\cos^4 s \cos^4  \phi_1$.

\subsection{Spherical coordinates.}
$(x_1, y_1, x_2, y_2)= (\cos s, \sin s(\cos t, \sin t \, e^{\ii \chi}))$,  $s,t \in [0, \pi]$, $\chi \in [0, 2\pi)$.
The above standard orthonormal frame becomes
\begin{equation} \label{susp}
\begin{split}
\xi &= \cos t \partial_{s} - \cot s \sin t \partial_{t} +\partial_{\chi}, \\
X_1 &= \sin t \cos  \chi \partial_{s} +(\cot s \cos t  \cos  \chi - \sin \chi)\partial_{t} -(\cot t \cos \chi+ \cot s \csc t \sin \chi)\partial_{\chi}, \\
X_2 &= \sin t \sin \chi \partial_{s} +(\cot s \cos t  \sin  \chi + \cos \chi)\partial_{t} -(\cot t \sin \chi - \cot s \csc t \cos \chi)\partial_{\chi} .
\end{split}
\end{equation}

\end{document}